\documentclass[12pt,reqno]{amsart}
\usepackage{fullpage}
\usepackage{times}
\usepackage{amsmath,amssymb,amsthm,url}
\usepackage[utf8]{inputenc}
\usepackage[english]{babel}
\usepackage{comment}
\usepackage{bbm}
\usepackage{enumerate}
\usepackage{bm}
\usepackage{graphicx}
\usepackage{mathrsfs}
\usepackage{mathtools}
\newenvironment{poc}{\begin{proof}[Proof of claim]}{\end{proof}}

\usepackage[colorlinks=true, pdfstartview=FitH, linkcolor=blue, citecolor=blue, urlcolor=blue]{hyperref}
\usepackage[vlined,ruled]{algorithm2e}
\DeclarePairedDelimiter\abs{\lvert}{\rvert}

\newtheorem{thm}{Theorem}[section]
\newtheorem{lem}[thm]{Lemma}
\newtheorem{prop}[thm]{Proposition}
\newtheorem{cor}[thm]{Corollary}
\newtheorem{conj}[thm]{Conjecture}

\theoremstyle{definition}

\newtheorem{question}[thm]{Question}

\newtheorem{rem}[thm]{Remark}

\newtheorem{claim}[thm]{Claim}

\newcommand{\ep}{\varepsilon}

\newcommand{\N}{\mathbb{N}}

\newcommand{\F}{\mathbb{F}}
\newcommand{\rad}{\operatorname{rad}}

\title{Bipartite Diophantine tuples and their applications}
\author{Kin Ming Tsang}
\address{Department of Mathematics \\ University of British Columbia\\ Vancouver  V6T 1Z2 \\ Canada}
\email{kmtsang@math.ubc.ca}
\author{Chi Hoi Yip}
\address{School of Mathematics\\ Georgia Institute of Technology\\ Atlanta, GA 30332\\ United States}
\email{cyip30@gatech.edu}
\begin{document}

\begin{abstract}
This paper investigates bipartite variants of generalized Diophantine tuples and their applications. We generalize a result of Bugeaud--Dujella on a special family of bipartite Diophantine tuples and affirmatively resolve a related question posed by the second author. Additionally, we establish new connections between bipartite Diophantine tuples and several known variants of Diophantine tuples, including those introduced by Banks--Luca--Szalay and Kihel--Kihel.
\end{abstract} 

\maketitle

\section{Introduction}
A set $\{a_{1}, a_{2},\ldots, a_{m}\}$ of distinct positive integers is a \textit{Diophantine $m$-tuple} if $a_ia_j+1$ is a perfect square for all $1\leq i<j\leq m$. Many generalizations and variants of Diophantine tuples have been studied extensively. We refer to the recent book of Dujella \cite{D24} for a comprehensive overview of the topic. In this paper, we prove several new results on bipartite variants of Diophantine tuples and discuss how bipartite Diophantine tuples are closely related to several variants of Diophantine tuples studied in the literature. In particular, we will discuss the connection between bipartite Diophantine tuples and variants of Diophantine tuples introduced by Banks--Luca--Szalay \cite{BLS09} and Kihel--Kihel \cite{KK01}.

Throughout the paper, let $k,n$ be integers with $k\geq 2$ and $n\neq 0$, and let $\N$ be the set of positive integers. A set $A$ of positive integers is a \textit{Diophantine tuple with property $D_{k}(n)$} if the product of $ab+n$ is a perfect $k$-th power for every $a,b\in A$ with $a\neq b$. Following the standard notations, we also write
\[M_{k}(n)=\sup \{|A| \colon A\subseteq{\mathbb{N}} \text{ satisfies the property }D_{k}(n)\}.\] 
These natural generalized notions of Diophantine tuples have been studied extensively; see, for example \cite{BDHL11, BD03, DKM22, D02, KYY, Y24+, Y24}. The best-known upper bound on $M_k(n)$ is of the form $M_k(n)\ll_k \log (|n|+1)$; see \cite{KYY, Y24+, Y24} for the best-known implied constant depending on $k$. Intuitively, a Diophantine tuple with property $D_{k}(n)$ is very restrictive, so there should be a uniform upper bound on $M_k(n)$. Under the Uniformity Conjecture \cite{CHM} (a consequence of the Bombieri–Lang conjecture), it is well-known that for each $k \geq 2$, there is a constant $C_k$ such that $M_k(n)\leq C_k$ holds for all nonzero integers $n$; see, for example, \cite{CY25, D02, KYY}. Assuming both the Uniformity Conjecture and the Lander--Parkin--Selfridge conjecture~\cite{LPS67} on sums of powers, recently Croot and the second author \cite{CY25} showed that there is an absolute constant $C$ such that $M_k(n)\leq C$ holds for all $k\geq 2$ and $n \neq 0$.

The second author recently introduced bipartite Diophantine tuples \cite{Y24}. However, the same objects have been studied, for example, by Bugeaud and Dujella \cite{BD03} and Bugeaud and Gyarmati \cite{BG04} two decades ago. More precisely, following \cite{Y24}, for each $k \ge 2$ and each nonzero integer $n$, we call a pair of sets $(A, B)$ a \textit{bipartite Diophantine tuple with property $BD_{k}(n)$} if $A, B$ are two subsets of $\N$ with size at least $2$, such that $ab+n$ is a $k$-th power for each $a \in A$ and $b \in B$. 

As observed in \cite{Y24}, Diophantine tuples are special bipartite Diophantine tuples, more precisely, if $A$ is a Diophantine tuple with property $D_{k}(n)$, then for any partition of $A$ into two subsets $A_1$ and $A_2$ each with size at least $2$, $(A_1, A_2)$ forms a bipartite Diophantine tuple with property $BD_{k}(n)$. Generally speaking, bipartite Diophantine tuples are much harder to study compared to Diophantine tuples, since being a Diophantine tuple imposes many more restrictions. For example, the quantities $M_2(1)$ and $M_2(-1)$ were studied extensively. Eventually, He, Togb\'e, and Ziegler \cite{HTZ19} proved that $M_2(1)=4$, and Bonciocat, Cipu, and Mignotte \cite{BCM22} proved that $M_2(-1)=3$. Nevertheless, it remains an open question to show that if $(A, B)$ is a bipartite Diophantine tuple with property $BD_{2}(1)$ (or $BD_{2}(-1)$, resp.), then $\min \{|A|,|B|\}$ is bounded by an absolute constant \cite{BHP25, CY25, Y26}; and the more general question for $BD_2(n)$ with $|n|\geq 2$ appears to be even harder. The second author \cite{Y24} studied the same question for $k\geq 3$ and proved the following theorem.

\begin{thm}[{\cite[Theorem 2.2]{Y24}}]\label{thm:logn}
If $k,n$ are integers with $k\geq 3$ and $n\neq 0$, then $$\min \{|A|,|B|\}\ll_{k} \log (|n|+1)$$ holds for all bipartite Diophantine tuples $(A, B)$ with property $BD_{k}(n)$.    
\end{thm}

Bugeaud and Dujella \cite{BD03} proved bounds on special bipartite Diophantine tuples and used them to deduce bounds on Diophantine tuples. They proved that if $k\geq 4$, and $a_1<a_2<b_1<b_2<\cdots <b_m$ are positive integers such that $a_ib_j+1$ are $k$-th powers for $1\leq i \leq 2$ and $1\leq j \leq m$, then $m\leq C_1(k)$, where $C_1(4)=3$, $C_1(k)=2$ for $5\leq k \leq 176$, and $C_1(k)=1$ for $k\geq 177$. They also proved that if $a_1<a_2<a_3<b_1<b_2<\cdots <b_m$ are positive integers such that $a_ib_j+1$ are perfect cubes for $1\leq i \leq 3$ and $1\leq j \leq m$, then $m\leq 6$. Equivalently, they studied bipartite Diophantine tuples $(A, B)$ with property $BD_k(1)$ with the additional constraint that the maximum element in $A$ is less than the minimum element in $B$. Our first main result extends their results to bipartite Diophantine tuples with property $BD_k(n)$ with $n\neq 0$.

\begin{thm}\label{thm1}
Let $\ep>0$ and $k,n$ be integers with $n\neq 0$. If $k\geq 4$, and $a_1<a_2\leq b_1<b_2<\cdots <b_m$ are positive integers such that $a_ib_j+n$ are $k$-th powers for $1\leq i \leq 2$ and $1\leq j \leq m$, then 
$$m\ll_{k,\ep} |n|^{\frac{\phi(k)}{(k-3)k}+\ep},$$
where $\phi$ is Euler's totient function. If $a_1<a_2<a_3\leq b_1<b_2<\cdots <b_m$ are positive integers such that $a_ib_j+n$ are perfect cubes for $1\leq i \leq 3$ and $1\leq j \leq m$, then $$m\ll_{\ep} |n|^{\frac{34}{27}+\ep}.$$
\end{thm}

Naturally, one may ask what would happen if the condition on the ordering of the elements in $A$ and $B$ is dropped. Equivalently, we can ask the following question.
\begin{question}\label{q}
Fix $k\geq 2$ and $n\neq 0$. What is the smallest integer $\ell$, such that if $(A,B)$ is a bipartite Diophantine tuple with property $BD_k(n)$ and $|A|=\ell$, then $|B|$ is absolutely bounded (possibly as a function of $k$ and $n$)? 
\end{question}

It is shown in \cite[Theorem 2.3]{Y24} that, for Question~\ref{q}, we have $\ell\leq 9$ for $k=3$, $\ell\leq 6$ for $k=4$, $\ell\leq 5$ for $k=5$, and $\ell\leq 4$ for $k\geq 6$. However, as we discussed earlier, when $k=2$, we do not know any upper bound on $\ell$. In~\cite[Remark 4.9]{Y24}, the second author asked if some improved bounds can be proved in Question~\ref{q} under the ABC conjecture. We answer this question in the affirmative in Theorem~\ref{thm:ABC} below; in particular, we show that $\ell=2$ if $k\geq 6$, $\ell\leq 3$ if $k\in \{4,5\}$, and $\ell\leq 5$ if $k=3$.

\begin{thm}\label{thm:ABC}
Assume the ABC conjecture. Let $(A,B)$ be a bipartite Diophantine tuple with property $BD_k(n)$ for some integers $n$ and $k$ with $n\neq 0$ and $k\geq 3$.
\begin{enumerate}
    \item If $|A|=2$ and $k\geq 6$, then for each $\ep>0$, 
    $$
|B|\ll_{k,\ep} |n|^{\frac{(k-1)\phi(k)}{k(k^2-6k+3)}+\ep}.
$$
 \item If $|A|=3$, then for each $\ep>0$, we have $|B|\ll_{\ep} |n|^{\frac{3}{8}+\ep}$ for $k=4$, and $|B|\ll_{\ep} |n|^{\frac{64}{625}+\ep}$ for $k=5$.
  \item If $|A|=5$ and $k=3$, then for each $\ep>0$, we have $|B|\ll_{\ep} |n|^{\frac{32}{81}+\ep}$.
\end{enumerate}
\end{thm}

We remark that there are some known bounds for Question~\ref{q} under other conjectures. As remarked in \cite[Remark 4.9]{Y24}, if $k\geq 5$, we can predict from the uniformity conjecture \cite{CHM} that the answer to Question~\ref{q} is $\ell=2$ by considering superelliptic curves of the form $C:y^k=(a_1x+n)(a_2x+n)$ with $a_1\neq a_2$. 
Similarly, using the uniformity conjecture, for $k \in \{3,4\}$, we can predict that $\ell \leq 3$ by considering superelliptic curves of the form $C: y^k=(a_1 x +n)(a_2 x +n) (a_3 x+n)$ with $a_1,a_2,a_3$ distinct; for $k=2$, we can predict that $\ell \leq 5$ by considering hyperelliptic curves of the form
$$
C: y^2=(a_1x+n)(a_2x+n)(a_3x+n)(a_4x+n)(a_5x+n)
$$
with $a_1,a_2,a_3,a_4,a_5$ distinct. On the other hand, Croot and the second author \cite[Theorem 2.10]{CY25} showed that a special case of Lander--Parkin--Selfridge conjecture~\cite{LPS67} on sums of $k$-th powers (see \cite[Conjecture 2.9]{CY25}) implies that the answer to Question~\ref{q} is $\ell=2$ for $k\geq 25$. 

\medskip

While bipartite Diophantine tuples are of independent interest, they have also played important roles in recent studies of other variants of Diophantine tuples. For example, Croot and the second author \cite{CY25} recently found that they are curial in the study of ``Diophantine powersets" (namely, a set $A$ of positive integers such that $ab+n$ is a \emph{perfect power} for all $a,b\in A$ with $a\neq b$, where $n\neq 0$ is a fixed integer). They are also key ingredients in the proof of a conjecture of Hajdu and S\'{a}rk\"{o}zy on the multiplicative irreducibility of all ``small perturbations" of the set of shifted $k$-th powers \cite{HS20}, recently established for $k\geq 3$ by the second author \cite{Y26}. Next, we discuss some new connections. 

A. Kihel and O. Kihel \cite{KK01} considered a different generalization of Diophantine tuples. They defined a $P_n^{(k)}$-set of to be a set $A=\{a_1,a_2,\ldots, a_m\}$ of positive integers such that $\prod_{j \in J} a_j \ + n$ is a perfect $k$-th power for each subset $J$ of $\{1,2,\ldots, m\}$ with size $k$. They proved that any $P_n^{(k)}$-set is finite. Similar generalizations of Diophantine tuples have been studied over finite fields; see, for example, \cite{KYY25, YY25}. In particular, in \cite{YY25}, Yoo and the second author showed that if $q$ is an odd prime power and $\ell \geq 2$ is an integer, then there is a subset $A$ of the finite field $\F_q$ with size $|A|\gg (\log q)/\ell$ such that $\prod_{a \in B}a\ +1$ is a square $\F_q$ for each subset $B$ of $A$ with size $\ell$. Inspired by these papers, we prove the following result, which, in particular, makes the result in \cite{KK01} effective.

\begin{thm}\label{thm:prodl}
Let $k,n,\ell$ be integers with $\ell \geq 2, k\geq 3$ and $n\neq 0$. If $A$ is a set of positive integers such that there exists a positive integer $a_0$, such that 
$a_0\prod_{a \in B}a\ +n$ is a perfect $k$-th power for each subset $B$ of $A$ with size $\ell$, then $$|A|\ll_{k} \ell+ \frac{\log (|n|+1)}{\ell}.$$
\end{thm}
Note that in the case $\ell=2$ and $a_0=1$, such a set $A$ is precisely a Diophantine tuple with property $D_k(n)$; when $\ell=k$ and $a_0=1$, such a set $A$ is precisely a $P_n^{(k)}$-set. 

\medskip

We also explore the connection between bipartite Diophantine tuples and multiplicative Hilbert cubes in shifted $k$-th powers. This is motivated by a variant of Diophantine tuples studied by Banks, Luca, and Szalay \cite{BLS09}. Given $A \subseteq \N$, they say that $S \subseteq \N$ is \emph{strongly $A$-Diophantine} if $1+\prod_{n\in R}n \in A$ for every finite subset $R \subseteq S$, where the product of elements in an empty set is defined to be $1$. We also refer to a relevant recent paper by Dujella and Szalay \cite{DS25}, where they showed that there are infinitely many integers $a,b,c\geq 2$ such that $ab+1,bc+1,ca+1,abc+1$ are all perfect squares. On the other hand, Bugeaud~\cite{B04} and Bennett~\cite{B07} studied whether it is possible for $a+1,ab+1, ab^2+1$ to be all $k$-th powers for some fixed $k$, where $a,b$ are positive integers with $a\geq 2$. They also studied the same question for $a+1, b+1, ab+1$, and other variants.

Motivated by these variants of Diophantine tuples studied in the literature, we study sets $A$ such that all subset products are shifted $k$-th powers, and more generally, multiplicative Hilbert cubes contained in shifted $k$-th powers. 
Multiplicative Hilbert cubes have been studied extensively; see, for example \cite{H08, S22}. Let $a_0$ be a positive integer and let $a_1,a_2,\ldots, a_d$ be distinct positive integers; the \emph{multiplicative Hilbert cube}
 $$
H^\times (a_0;a_1,a_2,\ldots, a_d)=\bigg\{a_0 \prod_{i \in I} a_i: I \subseteq \{1,2,\ldots, d\}\bigg\}
$$
is said to have dimension $d$. We are interested in the following question on bounding the maximum possible dimension of a multiplicative Hilbert cube contained in shifted $k$-th powers.

\begin{question}\label{Q2}
Given integers $k,n$ with $k\geq 2$ and $n\neq 0$, what's the maximum dimension of a multiplicative Hilbert cube contained in $\{x^k-n: x\in \N\}$? 
\end{question}

Question~\ref{Q2} essentially asks how to measure the largest possible multiplicative complexity of the set of shifted $k$-th powers. We expect that the maximum possible dimension of multiplicative Hilbert cubes is much smaller than that of bipartite Diophantine tuples or $P_{n}^{(k)}$ sets, as the defining conditions for multiplicative Hilbert cubes impose substantially more multiplicative constraints. 

As another motivation, Question~\ref{Q2} is a multiplicative analogue of a well-known question on (additive) Hilbert cubes contained in perfect $k$-th powers. To be precise, if $a_0$ is a nonnegative integer, and $a_1,a_2,\ldots, a_d$ are positive integers, then we   
    define the \emph{Hilbert cube} 
$$
H(a_0;a_1,a_2,\ldots, a_d)=\bigg\{a_0+ \sum_{i \in I} a_i: I \subseteq \{1,2,\ldots, d\}\bigg\}.
$$
The best-known bound is due to Dietmann and Elsholtz \cite{DE12, DE15}, where they showed that if $k\geq 2$, and $H(a_0;a_1,a_2,\ldots, a_d)\subseteq \{1,2,\ldots, N\} \cap \{x^k: x\in \N\}$, then $d\ll_{k} \log \log N$.

Now we discuss the connection of Question~\ref{Q2} and bipartite Diophantine tuples. Assume $H^\times (a_0;a_1,a_2,\ldots, a_d) \subseteq \{x^k-n: x\in \N\}$. Note that if $a_0=1$ and $A=\{a_1,a_2,\ldots, a_d\}$, then $H^\times (a_0;a_1,a_2,\ldots, a_d)$ is the collection of all subset products of $A$; in particular, $A$ is a Diophantine tuple with property $D_k(n)$. In general, observe that $(\{a_0a_1, a_0a_2, \ldots, a_0a_{\lfloor d/2\rfloor}\}, \{a_{\lfloor d/2\rfloor+1}, \ldots, a_d\})$ forms a bipartite Diophantine tuple with property $BD_k(n)$. Thus, for $k\geq 3$, it follows immediately from Theorem~\ref{thm:logn} that $d\ll_{k}\log (|n|+1)$. One hopes that this ``trivial" upper bound can be improved significantly. Indeed, for $k\geq 3$, by taking $\ell=\sqrt{\log (|n|+1)}$, Theorem~\ref{thm:prodl} immediately implies the following corollary.

\begin{cor}\label{cor:cube}
Let $k,n$ be integers with $k\geq 3$ and $n\neq 0$. Assume that $a_0$ is a positive integer, and $a_1,a_2,\ldots, a_d$ are distinct positive integers such that $H^\times (a_0;a_1,a_2,\ldots, a_d) \subseteq \{x^k-n: x \in \N\}$. Then we have $d\ll_{k} \sqrt{\log (|n|+1)}$.    
\end{cor}

We also prove some much better bounds under the additional assumption that $a_0$ is small. The next theorem addresses the case that $a_0=1$. Here, we did not attempt to optimize the bound, and a better bound can likely be achieved with extra work.

\begin{thm}\label{thm:a0=1}
Let $k,n$ be integers with $k\geq 2$ and $n\geq 1$. If $A$ is a set of positive integers such that $\prod_{a \in B}a\ +n$ is a perfect $k$-th power for each subset $B$ of $A$, then $|A|\leq 132$ for $k=2$ and $|A|\leq 9$ for $k\geq 3$.
\end{thm}

We also prove the following result that addresses the case when $a_0$ is small compared to $n$.

\begin{thm}\label{thm:a0small}
Let $\ep \in (0,1)$ be a real number. Assume that $H^\times (a_0;a_1,a_2,\ldots, a_d) \subseteq \{x^k-n: x \in \N\}$, where $k,n,a_0$ are positive integers with $k\geq 2$ and $a_0\leq n^{\frac{k-1}{k}-\ep}$, and $a_1, \ldots, a_d$ are distinct positive integers. Then $d\ll \frac{1}{\ep}$, where the implied constant is absolute.
\end{thm}

When $k=2$, we are not able to prove a result similar to Corollary~\ref{cor:cube}, where the upper bound on the dimension of multiplicative Hilbert cubes only depends on the shift. This is mainly due to the fact, mentioned earlier, that it is an open problem to show that for each $n\neq 0$, there is an absolute upper bound on $\min \{|A|, |B|\}$ for a bipartite Diophantine tuple $(A,B)$ with property $BD_2(n)$. Nevertheless, we prove the following theorem, which bounds the dimension of a multiplicative Hilbert cube contained in shifted squares in terms of its smallest generator. Such a result might be useful in estimating multiplicative Hilbert cubes, contained in shifted squares, that are obtained by extending smaller multiplicative Hilbert cubes, or with bounded generators.

\begin{thm}\label{thm:a1}
Assume that $H^\times (a_0;a_1,a_2,\ldots, a_d) \subseteq \{x^2-n: x \in \N\}$, where $a_0, a_1, \ldots, a_d$ are positive integers with $2\leq a_1<a_2<\ldots<a_d$, and $n$ is a nonzero integer. Then 
$$
d\ll 2^{\frac{\omega(n(1-a_{1}))}{2}}+\log |n|,
$$
where $\omega(n(1-a_{1}))$ counts the number of distinct prime factors of $n(1-a_1)$. In particular, 
$$d\leq 2^\frac{(1+o(1))\log (a_1|n|)}{2\log \log (a_1|n|)}.$$  
\end{thm}

Our proofs of these new results combine various tools from additive combinatorics, Diophantine approximation, Diophantine equations, and sieve methods. Some of the ideas are inspired by \cite{BDHL11, CY25, D02, Y24}. Our proofs of the above results related to Hilbert cubes crucially depend on the many multiplicative constraints imposed by the defining conditions of multiplicative Hilbert cubes.

\medskip

\textbf{Notation.} In this paper,~$p$ always denotes a prime, and $\sum_p$ represents sums over all primes. We follow the Vinogradov notation $\ll$. We write $X \ll Y$ if there is an absolute constant $C>0$ so that $|X| \leq CY$. 

\textbf{Organization of the paper.} In Section~\ref{sec:prelim}, we prove several useful results related to the gap principle, sieve methods, and additive combinatorics. In Section~\ref{sec:antigap}, we prove an anti-gap principle for bipartite Diophantine tuples with property $BD_2(n)$. In Section~\ref{sec:main1}, we prove Theorems~\ref{thm1} and~\ref{thm:ABC}. In Section~\ref{sec:app}, we combine the tools developed in Section~\ref{sec:prelim} and Section~\ref{sec:antigap} to prove Theorems~\ref{thm:prodl},~\ref{thm:a0=1}, and~\ref{thm:a0small}. Finally, in Section~\ref{sec:Pell}, we prove Theorem~\ref{thm:a1}.

\section{Preliminaries}\label{sec:prelim}

\subsection{A gap principle and its applications}

The following lemma describes a useful gap principle.

\begin{lem}[{\cite[Lemma 3.6]{Y24}}]\label{gap_principle}
Let $k \geq 3$ and $n$ be a nonzero integer. Let $x,y,z,w$ are positive integers such that $x<y$, $z<w$, and $xz\geq 2|n|$. Suppose further that $xz+n, yz+n, xw+n, yw+n$ are $k$-th powers. Then $yw \geq k^{k} (xz)^{k-1}/(4^{k-1}|n|^k)$.     
\end{lem}

Next, we deduce a few useful corollaries.

\begin{cor}\label{cor:gap}
Let $k \geq 4$ and $n$ be a nonzero integer. Let $L$ be a real number such that $L>\frac{k}{k-3}$. Let $A=\{a_1,a_2\}$ and $B=\{b_1, b_2, \ldots, b_m\}$ be subsets of positive integers such that $a_1<a_2\leq b_1<b_2<\cdots<b_m$, and $ab+n$ is a $k$-th power for each $a \in A$ and $b \in B$. If $b_1\geq 2|n|^{L}$, then $b_i\geq b_1^{\theta^{i-1}}$ for each $1 \leq i \leq m$, where $\theta=k-2-\frac{k}{L}>1$.     
\end{cor}
\begin{proof}
Let $1\leq i<m$. It suffices to show that $b_{i+1}\geq b_i^{\theta}$. Indeed, since $b_1 \geq 2|n|$ and $b_1\geq a_2$, applying Lemma~\ref{gap_principle} with $x=a_1,y=a_2, z=b_i, w=b_{i+1}$, we have
$$
b_{i+1}\geq \frac{k^{k} (a_1b_i)^{k-1}}{a_2 \cdot 4^{k-1}|n|^k} \geq
\frac{b_i^{k-2}}{|n|^k}> \frac{b_i^{k-2}}{b_i^{k/L}}=b_i^{k-2-\frac{k}{L}}=b_i^{\theta},
$$
as required.
\end{proof}

\begin{cor}\label{cor:a0X}
Let $k,n$ be integers with $k\geq 3$ and $n\neq 0$. Assume that $H^\times (a_0;a_1,a_2,\ldots, a_d) \subseteq \{x^k-n: x \in \N\}$, where $a_1<a_2<\ldots<a_d$ and $d\geq 10$ is even. Then $a_0X<|n|^{k/(k-2)}$, where $X=\prod_{j=d/2+1}^{d} a_j$.
\end{cor}

\begin{proof}
Assume that $a_0X\geq |n|^{k/(k-2)}$. Since $d\geq 10$, observe that $$(\{a_0a_1a_5X, a_0a_2a_5X\}, \{a_3,a_4\})$$ forms a bipartite Diophantine tuple with property $BD_k(n)$.  Thus, by applying Lemma~\ref{gap_principle} with $x=a_0a_1a_5X$, $y=a_0a_2a_5X$, $z=a_3$ and $w=a_4$, we obtain
$$
a_0a_2a_4a_5X \geq \frac{k^k (a_0a_1a_3a_5X)^{k-1}}{4^{k-1}|n|^k}.
$$
It follows that
$$
|n|^k\geq \frac{4^{k-1}|n|^k}{k^k}\geq \frac{(a_0a_1a_3a_5X)^{k-1}}{a_0a_2a_4a_5X}\geq \frac{a_0^{k-2}a_5^{k-2}X^{k-2}}{a_4}\geq (a_0X)^{k-2},
$$
that is, $a_0X<|n|^{k/(k-2)}$. 
\end{proof}

The following proposition is a slightly weaker version of \cite[Proposition 4.1]{Y24}. Its proof is based on an explicit bound on the number of solutions of Thue inequalities by Evertse \cite{E83} combined with repeated applications of Lemma~\ref{gap_principle}. 

\begin{prop}[{\cite[Proposition 4.1]{Y24}}]\label{prop:a1a2}
Let $k \geq 3$ and $n$ be a nonzero integer. Let $A,B \subseteq \N$ such that $A=\{a_1, a_2, \ldots, a_{\ell}\}$ and $B=\{b_1, b_2, \ldots, b_m\}$ with $a_1<a_2<\cdots<a_{\ell}$, and $b_1<b_2<\cdots<b_m$, and $AB+n \subseteq \{x^k: x \in \N\}$. 
Assume that $m\geq 7$, $\ell \geq 2$, and $a_2 \leq b_{m-6}$. If $k=3$, further assume that $\ell \geq 3$, and $a_3 \leq b_{m-6}$. Then at most $6$ elements in $B$ are at least $2|n|^{17}$. 
\end{prop}

\subsection{Applications of sieve methods}
In this subsection, we use sieve methods to bound the size of certain bipartite Diophantine tuples. We first recall Gallagher's larger sieve \cite{G71}.
\begin{lem}[Gallagher's larger sieve]\label{GS}  
Let $N$ be a positive integer and $A\subseteq\{1,2,\ldots, N\}$. Let ${\mathcal P}$ be a set of primes. For each prime $p \in {\mathcal P}$, let $A_p=A \pmod{p}$. For any $1<Q\leq N$, we have
$$
 |A|\leq \frac{\underset{p\leq Q, \ p\in \mathcal{P}}\sum\log p - \log N}{\underset{p\leq Q, \ p \in \mathcal{P}}\sum\frac{\log p}{|A_p|}-\log N},
$$
provided that the denominator is positive.
\end{lem}

To apply the above larger sieve, we need to study a finite field model of Diophantine tuples. As usual, for a prime number $p$, let $\F_p$ be the finite field with $p$ elements and $\F_p^*=\F_p \setminus \{0\}$. The following lemma can be deduced via a standard application of Weil's bound on complete character sums (see, for example \cite[Exercise 5.66]{LN97}).

\begin{lem}[{\cite[Lemma 5.3]{CY25}}]\label{lem:Weil}
Let $k \geq 2$ and $p$ be a prime such that $p \equiv 1 \pmod k$. Let $A_p \subseteq \F_p^*$, $B_p \subseteq \F_p$, and $\lambda \in \F_p^*$, such that $ab+\lambda \in \{x^k: x \in \F_p\}$ for all $a\in A_p$ and $b\in B_p$. If $|A_p|=m$, then $|B_p|\leq \frac{p}{k^m}+m\sqrt{p}$.
\end{lem}

A special case of the following proposition has appeared implicitly in the proof of \cite[Theorem 2.5]{CY25}. Here we include a short proof for the sake of completeness.

\begin{prop}\label{prop:sieve2}
Let $k,m$ be fixed positive integers with $k\geq 2$, and $C, L$ be fixed positive real numbers with $L>1$. 
If $A$ is a subset of $\{1,2,\ldots, M\}$ with $|A|=m$, $B$ is a subset of $\{1,2,\ldots, N\}$ such that $(A, B)$ forms a bipartite Diophantine tuple with property $BD_k(n)$, where $0<|n|\leq N$, and $M\leq CN^{L}$, then for each $\ep>0$, we have $$|B|\ll_{\ep,k,m,C,L} N^{\frac{\phi(k)}{k^m}+\ep}.$$ 
\end{prop}
\begin{proof}
Let $\delta=\frac{k^m\ep}{\phi(k)}$. Consider the set of primes $\mathcal{P}_0=\{p\leq Q: p \equiv 1 \pmod k\}$, where $Q=N^{\frac{(1+\delta)\phi(k)}{k^m}}$. For each $p \in \mathcal{P}_0$, let $A_p$ be the image of $A$ modulo $p$ and view $A_p$ as a subset of  $\F_p$; similarly, define $B_p$. 

Consider the following subset of $\mathcal{P}_0$:
$$\mathcal{P}=\{p\in \mathcal{P}_0: \quad|A_p|=m, \quad 0 \notin A_p, \quad p \nmid n\}.$$
For each $p \in \mathcal{P}$, we have $A_p \subseteq \F_p^*$ with $|A_p|=m$, $p \nmid n$, and $ab+n \in \{x^k: x \in \F_p\}$ for each $a\in A_p$ and $b\in B_p$, and thus Lemma~\ref{lem:Weil} implies that $|B_p|\leq \frac{p}{k^m}+m\sqrt{p}.$ 

On the other hand, note that the number of ``bad" primes $p \in \mathcal{P}_0$, that is, primes $p$ such that one of $|A_p|<m$, $0\in A_p$, and $p\mid n$ holds, is $\ll_{m, C, L} \log N$, since the number of prime divisors of a positive integer $\leq MN$ is $\ll_{C, L} \log N$. Thus, $|\mathcal{P}_0 \setminus \mathcal{P}|\ll_{m,C,L} \log N$ and \cite[Lemma 3.10]{CY25} implies that
$$
\sum_{p \in \mathcal{P}_0 \setminus \mathcal{P}} \frac{\log p}{p} \ll_{m,C,L} \log \log N.
$$
Thus, 
\begin{equation}\label{eq:logQ1}
\sum_{p \in \mathcal{P}_0 \setminus \mathcal{P}} \frac{\log p}{|B_p|}\geq \sum_{p \in \mathcal{P}_0 \setminus \mathcal{P}} \frac{\log p}{\frac{p}{k^m}+m\sqrt{p}} \geq\sum_{p \in \mathcal{P}_0} \frac{\log p}{\frac{p}{k^m}+m\sqrt{p}}\ -O_{k,m,C,L}(\log \log N).
\end{equation}
By the prime number theorem for arithmetic progressions, we have
\begin{equation}\label{eq:logQ2}
\sum_{p \in \mathcal{P}_0} \frac{\log p}{p} \sim \frac{\log Q}{\phi(k)}.
\end{equation}
Then, when $N$ is sufficiently large compared to $k,m,C,L,\delta$ so that $Q$ is sufficiently large, estimates~\eqref{eq:logQ1} and~\eqref{eq:logQ2} imply that
$$
\sum_{p \in \mathcal{P}_0 \setminus \mathcal{P}} \frac{\log p}{|B_p|} \geq \frac{2+\delta}{2+2\delta} \frac{k^m\log Q}{\phi(k)}=\bigg(1+\frac{\delta}{2}\bigg)\log N.
$$
Thus, when $N$ is sufficiently large, Gallagher's sieve (Lemma~\ref{GS}) implies that
\[ 
|B|\le \frac{\sum_{p \in \mathcal{P}}\log p - \log N}{\sum_{p \in \mathcal{P}} \frac{\log p}{|B_{p}|}-\log N}\leq \frac{4Q}{\delta \log N}\ll_{\delta} Q\ll_{\ep} N^{\frac{\phi(k)}{k^m}+\ep},
\]
as required.
\end{proof}

\subsection{A lemma from additive combinatorics}

Let $S$ be a finite subset of real numbers, and $h$ be a positive integer, the \emph{$h$-fold restricted sumset of $S$} is
$$
h^{\wedge}S=\{s_1+s_2+\cdots+s_h: s_1, s_2, \ldots, s_h \text{ are distinct elements in } S\}.
$$
A special case of the main result of Dias da Silva and Hamidoune \cite{DH94} implies that $|h^{\wedge} S|\geq h(|S|-h)+1$. For a finite subset $A$ of positive integers, the \emph{$h$-fold restricted product set of $A$}
is
$$
A^{(h)}=\{a_1a_2\ldots a_h: a_1,a_2,\ldots, a_h \text{ are distinct elements in } A\}.
$$
Let $B=\{\log a: a \in A\}$; then $B$ is a subset of real numbers and we have
$$
A^{(h)}=\{\exp(s): s\in h^{\wedge}B\}.
$$
Thus, we have proved the following lemma.
\begin{lem}\label{lem:hfold}
Let $A$ be a finite subset of positive integers, and $h$ be a positive integer. Then $|A^{(h)}|\geq h(|A|-h)+1$.
\end{lem}

\section{An Anti-gap principle for bipartite Diophantine tuples with property $BD_2(n)$}\label{sec:antigap}

In this section, we establish the following anti-gap principle for bipartite Diophantine tuples with property $BD_2(n)$. In particular, it extends a result of Dujella \cite[Lemma 2]{D02} on Diophantine tuples with property $D_2(n)$.

\begin{prop}\label{prop:antigap}
Let $n$ be a nonzero integer. Assume that $a_{1} < a_{2} < a_{3}$ and $b_{1} < b_{2}$ are positive integers such that $(A,B)$ with $A = \{ a_{1}, a_{2}, a_{3} \}$ and $B = \{ b_{1}, b_{2} \}$ forms a bipartite Diophantine tuple with property $BD_2(n)$. If $b_{1} > a_{3}^{17}n^{10}$ and $a_3\geq7$, then $b_{2} < b_{1}^{120}$.
\end{prop}

A key ingredient of the proof is the following well-known theorem on simultaneous approximation of square roots of two rationals, due to Bennett \cite{B98}.

\begin{thm}[Bennett \cite{B98}]\label{thm:Thm5_v1}
If $c_{i}$, $p_{i}$, $q$ and $N$ are integers for $0 \leq i \leq 2$ with $c_{0} < c_{1} < c_{2}$, $c_{j} = 0$ for some $0 \leq j \leq 2$, $q \geq 1$ and $N > M^9$, where $M = \max_{0 \leq i \leq 2} \{ \abs{c_{i}} \} \geq 3$, then we have
\begin{align*}
    \max_{0 \leq i \leq 2} \left\{ \sqrt{1+ \frac{c_{i}}{N}} - \frac{p_{i}}{q} \right\} > (130 N \gamma)^{-1}q^{-\lambda}
\end{align*}
where 
\begin{align*}
    \lambda = 1 + \frac{\log \left (33N \gamma \right)}{\log \left( 1.7 N^{2} \prod_{0 \leq i < j \leq 2} (c_{i} - c_{j})^{-2} \right)}
\end{align*}
and 
\begin{equation}\label{eq:gamma}
    \gamma = 
    \begin{cases}
        \frac{(c_{2} - c_{0})^{2} (c_{2} - c_{1})^{2}}{2c_{2} - c_{0} - c_{1}} & \text{if } c_{2} - c_{1} \geq c_{1} - c_{0} \\
        \frac{(c_{2} - c_{0})^{2} (c_{1} - c_{0})^{2}}{c_{1} + c_{2} - 2c_{0}} & \text{if } c_{2} - c_{1} < c_{1} - c_{0}
    \end{cases}
\end{equation}

\end{thm}

\begin{rem}\label{rem:gamma}
It is convenient to estimate the quantity $\gamma$ (defined in equation~\eqref{eq:gamma}) by the following inequality 
$$
\frac{(c_2-c_0)^3}{6}<\gamma<\frac{(c_2-c_0)^3}{2}.
$$
Here we present a quick proof of the above inequality. Let $A=c_2-c_1$ and $B=c_1-c_0$. If $c_2-c_2\geq c_1-c_0$, that is, $A\geq B$, then
$$
\gamma=\frac{(c_{2} - c_{0})^{2} (c_{2} - c_{1})^{2}}{2c_{2} - c_{0} - c_{1}}=\frac{(A+B)^2A^2}{2A+B}
$$
and it follows that 
$$
\frac{(A+B)^3}{6}< (A+B)^3 \cdot \frac{A}{A+B} \cdot \frac{A}{2A+B}=
\gamma < (A+B)^3 \cdot \frac{A}{2A+B}< \frac{(A+B)^3}{2}.
$$
The proof for the case $A<B$ is similar.
\end{rem}

Next, we present the proof of Proposition~\ref{prop:antigap}.

\begin{proof}[Proof of Proposition~\ref{prop:antigap}]

By definition, there are positive integers $r,s,t,x,y,z$ such that
\begin{align*}
    a_{1}b_{1} + n &= r^{2}, & a_{2}b_{1} + n &= s^{2}, & a_{3}b_{1} + n &= t^{2}, \\
    a_{1}b_{2} + n &= x^{2}, & a_{2}b_{2} + n &= y^{2}, & a_{3}b_{2} + n &= z^{2} .
\end{align*}
Eliminating $b_{2}$, we obtain the following system of Pell equations
\begin{align}
    a_{1}z^{2} - a_{3}x^{2} &= n(a_{1} - a_{3}), \label{eq:13}\\
    a_{2}z^{2} - a_{3}y^{2} &= n(a_{2} - a_{3}).  \label{eq:23}
\end{align}
Let
\begin{align*}
    \theta_{1} &:= \frac{r\sqrt{a_{3}}}{t\sqrt{a_{1}}} = \sqrt{1 + \frac{n(a_{3}-a_{1})}{a_{1}a_{3}b_{1} + na_{1}}} = \sqrt{1 + \frac{n(a_{3}-a_{1})a_{2}}{a_{1}a_{2}t^{2}}},\\
    \theta_{2} &:= \frac{s\sqrt{a_{3}}}{t\sqrt{a_{2}}} = \sqrt{1+ \frac{n(a_{3}-a_{2})}{a_{2}a_{3}b_{1} + na_{2}}} = \sqrt{1+ \frac{n(a_{3}-a_{2})a_{1}}{a_{1}a_{2}t^{2}}}.
\end{align*}
We shall approximate $\theta_{1}$ by the rational number
\begin{align*}
    \frac{a_{3}rx}{a_{1}tz} = \frac{a_{2}a_{3}rx}{a_{1}a_{2}tz},
\end{align*}
and approximate $\theta_{2}$ by the rational number
\begin{align*}
    \frac{a_{3}sy}{a_{2}tz} = \frac{a_{1}a_{3}sy}{a_{1}a_{2}tz} .
\end{align*}

\begin{claim}\label{lem:Lem1_general_n}
We have
\begin{align}
    \max \left( \abs[\bigg]{\theta_{1} - \frac{a_{2}a_{3}rx}{a_{1}a_{2}tz}}, \abs[\bigg]{\theta_{2} - \frac{a_{1}a_{3}sy}{a_{1}a_{2}tz}} \right) < \frac{a_{3}\abs{n}}{a_{1}z^2}.
\end{align}
\end{claim}
\begin{poc}
By equation~\eqref{eq:13}, we have
$$
\abs[\bigg]{\theta_{1} - \frac{a_{2}a_{3}rx}{a_{1}a_{2}tz}}=\abs[\bigg]{\frac{r\sqrt{a_{3}}}{t\sqrt{a_{1}}} - \frac{a_{3}rx}{a_{1}tz}} = \frac{r\sqrt{a_{3}}}{a_{1}zt} \abs{z\sqrt{a_{1}} - x\sqrt{a_{3}}} = \frac{r\sqrt{a_{3}}}{a_{1}zt} \abs[\bigg]{\frac{n(a_{3}-a_{1})}{z\sqrt{a_{1}} + x\sqrt{a_{3}}}} .
$$
If $n < 0$, then $\theta_{1} = \frac{r\sqrt{a_{3}}}{t\sqrt{a_{1}}} = \sqrt{1 - \frac{\abs{n}(a_{3}-a_{1})a_{2}}{a_{1}a_{2}t^{2}}} < 1$ and we obtain
$$
\abs[\bigg]{\theta_{1} - \frac{a_{2}a_{3}rx}{a_{1}a_{2}tz}} < \frac{r\sqrt{a_{3}}\abs{n}a_{3}}{a_{1}\sqrt{a_{1}}tz^{2}} = \theta_1 \cdot  \frac{\abs{n}a_{3}}{a_{1}z^2} < \frac{a_{3}\abs{n}}{a_{1}z^2}.
$$
If $n > 0$, then $x\sqrt{a_{3}} > z\sqrt{a_{1}}$ by equation~\eqref{eq:13} and it follows that
$$
\abs[\bigg]{\theta_{1} - \frac{a_{2}a_{3}rx}{a_{1}a_{2}tz}} < \frac{\sqrt{a_{1}b_{1}+n}\sqrt{a_{3}}na_{3}}{2a_{1}\sqrt{a_{1}}z^{2}t} = \sqrt{1+\frac{n(a_{3}-a_{1})}{a_{1}t^{2}}}  \frac{a_{3}\abs{n}}{2a_{1}z^2} < \sqrt{1+\frac{na_3}{a_{1}a_3b_1}}  \frac{a_{3}\abs{n}}{2a_{1}z^2} <\frac{a_{3}\abs{n}}{a_{1}z^2},
$$
where we used the assumption that $b_1>n$ in the last step.

Similarly, we have
\[
\abs[\bigg]{\theta_{2} - \frac{a_{1}a_{3}sy}{a_{1}a_{2}tz}} < \frac{a_{3}\abs{n}}{a_{2}z^2} < \frac{a_3\abs{n}}{a_{1}z^2}.\qedhere
\]
\end{poc}

We apply Theorem \ref{thm:Thm5_v1} with $N = a_{1}a_{2}t^{2}$, $M = \abs{n(a_{3} - a_{1})a_{2}}$, $q = a_{1}a_{2}tz$, $c_1= n(a_{3} - a_{2})a_{1}$, and $p_{1} = a_{2}a_{3}rx$. Additionally, if $n>0$, we set 
$$
c_0=0, c_2=n(a_{3} - a_{1})a_{2}, p_0=q, p_{2} = a_{1}a_{3}sy;
$$
if $n<0$, we instead set
$$
c_2=0, c_0=n(a_{3} - a_{1})a_{2}, p_2=q, p_{0} = a_{1}a_{3}sy.
$$
Note that the assumption $b_{1} > a_{3}^{17}n^{10}$ implies that 
$$
N=a_1a_2t^2=a_1a_2(a_3b_1+n)>\frac{a_1a_2a_3b_1}{2}>\frac{a_1a_2a_3^{18} n^{10}}{2}>a_3^{18}n^{10}>|n|^9 a_3^9a_2^9>M^9.
$$
We have $c_2-c_0=|n|(a_3-a_1)a_2$ and thus the quantity $\gamma$ from Theorem~\ref{thm:Thm5_v1} satisfies 
\begin{align}\label{eq:gamma1}
    \frac{(a_{3}-a_{1})^{3}a_{2}^{3}}{6}\abs{n}^{3} < \gamma < \frac{(a_{3}-a_{1})^{3}a_{2}^{3}}{2}\abs{n}^{3}
\end{align}
by Remark~\ref{rem:gamma}. For the quantity $\lambda$ from Theorem \ref{thm:Thm5_v1} we have
\begin{align*}
    \lambda = 1 + \frac{\log(33a_{1}a_{2}t^{2} \gamma)}{\log(1.7t^{4}a_{3}^{-2}(a_{3}-a_{2})^{-2}(a_{3}-a_{1})^{-2}(a_{2}-a_{1})^{-2} n^{-6})} = 2-\lambda_{1},
\end{align*}
where
\begin{align}\label{eq:lambda1}
    \lambda_{1} = \frac{\log(\frac{1.7t^{2}}{33a_{1}a_{2}a_{3}^{2} (a_{3}-a_{2})^{2}(a_{3}-a_{1})^{2}(a_{2}-a_{1})^{2}n^{6}\gamma})}{\log(1.7t^{4}a_{3}^{-2}(a_{3}-a_{2})^{-2}(a_{3}-a_{1})^{-2}(a_{2}-a_{1})^{-2}n^{-6})}.
\end{align}
By inequality~\eqref{eq:gamma1}, we have
\begin{align*}
&33a_{1}a_{2}a_{3}^{2} (a_{3}-a_{2})^{2}(a_{3}-a_{1})^{2}(a_{2}-a_{1})^{2}n^{6}\gamma\\
&<\frac{33}{2} a_{1}a_{2}^4 a_{3}^{2} (a_{3}-a_{2})^{2}(a_{3}-a_{1})^{5}(a_{2}-a_{1})^{2}|n|^{9}\\
&<\frac{33}{2} a_{1}(a_{3}-a_{1})a_{2}^2 (a_{3}-a_{2})^{2}a_3^{10}|n|^{9}
<\frac{33}{128}a_3^{16}|n|^9,    
\end{align*}
where in the last step we used the inequality $a_j(a_3-a_j)\leq \frac{a_3^2}{4}$ that holds for $j\in \{1,2\}$. Since $b_1>a_3^{17}n^{10}$ and $a_3>7$, the above inequality implies that
\begin{equation}\label{eq:10log}
\frac{1.7t^{2}}{33a_{1}a_{2}a_{3}^{2} (a_{3}-a_{2})^{2}(a_{3}-a_{1})^{2}(a_{2}-a_{1})^{2}n^{6}\gamma}>\frac{\frac{1.7}{2}a_3b_1}{\frac{33}{128}a_3^{16}|n|^9}>\frac{b_1}{a_3^{15}|n|^9}>b_1^{\frac{1}{10}}.    
\end{equation}
Thus, $0 < \lambda_{1} < 1$ (i.e. $1 < \lambda < 2$).

Observe that
$$
 \max_{0 \leq i \leq 2} \left\{ \sqrt{1+ \frac{c_{i}}{N}} - \frac{p_{i}}{q} \right\}=\max \left( \abs[\bigg]{\theta_{1} - \frac{a_{2}a_{3}rx}{a_{1}a_{2}tz}}, \abs[\bigg]{\theta_{2} - \frac{a_{1}a_{3}sy}{a_{1}a_{2}tz}} \right).
$$
Thus Theorem \ref{thm:Thm5_v1} and Claim~\ref{lem:Lem1_general_n} imply that
\begin{align*}
    \frac{a_{3}\abs{n}}{a_{1}z^{2}} > (130a_{1}a_{2}t^{2}\gamma)^{-1}(a_{1}a_{2}tz)^{-\lambda},
\end{align*}
which implies
\begin{align*}
    z^{\lambda_{1}} < 130 a_{1}^{\lambda} a_{2}^{\lambda + 1} a_{3} t^{\lambda + 2} \abs{n} \gamma < 130 a_{1}^{2} a_{2}^{3} a_{3} t^{4} \abs{n} \gamma
\end{align*}
and thus equation~\eqref{eq:lambda1} implies that
\begin{align}\label{eq:logz_estimate_in_lem2_general_n}
    \log z < \frac{\log(130 a_{1}^{2} a_{2}^{3} a_{3} t^{4} \abs{n} \gamma) \log(1.7t^{4}a_{3}^{-2}(a_{3}-a_{2})^{-2}(a_{3}-a_{1})^{-2}(a_{2}-a_{1})^{-2}n^{-6})}{\log(\frac{1.7t^{2}}{33a_{1}a_{2}a_{3}^{2} (a_{3}-a_{2})^{2}(a_{3}-a_{1})^{2}(a_{2}-a_{1})^{2}n^{6}\gamma})}.
\end{align}

Let us estimate the right-hand side of inequality~\eqref{eq:logz_estimate_in_lem2_general_n}. Since $a_3\geq7$ and $b_1 > a_3^{17}n^{10}$, we have
\begin{align}\label{eq:3log}
    130 a_{1}^{2} a_{2}^{3} a_{3} t^{4} \abs{n} \gamma 
    &< 65 a_{1}^{2} a_{2}^{6} a_{3} (a_{3} - a_{1})^{3} (a_{3}b_{1}+n)^{2} n^{4} \notag\\
    &< 65 a_1^2 a_2^6 a_3^{4} (2a_3b_1 n)^2 n^4 <  260a_3^{14}n^6 b_1^2<  b_{1}^{3}.
\end{align}
Since $a_3\geq 7$, 
\begin{align}\label{eq:2log}
    &1.7t^{4}a_{3}^{-2}(a_{3}-a_{2})^{-2}(a_{3}-a_{1})^{-2}(a_{2}-a_{1})^{-2} n^{-6} \notag\\
    &\leq \frac{1.7(a_3b_1+n)^2}{4a_3^2n^6}\leq \frac{1.7}{4} \bigg(\frac{b_1}{|n|^3}+\frac{1}{a_3}\bigg)^2\leq \frac{1.7}{4} \bigg(b_1+\frac{b_1}{7}\bigg)^2<b_1^2.
\end{align}

Inserting the bounds in inequalities~\eqref{eq:10log},~\eqref{eq:3log},~\eqref{eq:2log} into the right-hand side of inequality~\eqref{eq:logz_estimate_in_lem2_general_n}, we obtain
\begin{align*}
    \log z < \frac{(3 \log b_{1})(2\log b_{1}) }{\frac{1}{10}\log b_{1}} = 60 \log b_{1} .
\end{align*}
Hence $z < b_{1}^{60}$. Since $a_3\geq 7$, we conclude that
\[
b_{2} = \frac{z^{2}-n}{a_{3}} \leq \frac{z^{2}+\abs{n}}{a_{3}} < \frac{b_{1}^{120} + b_{1}^{\frac{1}{10}}}{7} < b_{1}^{120}.\qedhere
\]
\end{proof}

As an application of Proposition~\ref{prop:antigap}, we deduce the following corollary on multiplicative Hilbert cubes contained in shifted squares.

\begin{cor}\label{cor:loga0}
Assume that $H^\times (a_0;a_1,a_2,\ldots, a_d) \subseteq \{x^2-n: x \in \N\}$ with $d\geq 7$, where $a_0, a_1, \ldots, a_d$ are positive integers with $a_1<a_2<\ldots<a_d$, and $n$ is a nonzero integer. Then $d\ll 1+\frac{\log a_0}{\log a_7}+\frac{\log \max \{n^{10}/a_0,1\}}{\log a_7}$, where the implied constant is absolute. % In particular, $d\ll \log a_0+\log |n|$.  
\end{cor}
\begin{proof}
Let $m$ be the smallest positive integer such that $a_0a_7^m>a_7^{17}n^{10}$; note that we have $m\ll 1+\frac{\log \max \{n^{10}/a_0,1\}}{\log a_7}$. If $d\leq m+7$, then we are done. Next, assume that $d\geq m+8$. Let
$$
b_1=a_0 \cdot \prod_{j=8}^{m+7} a_j.
$$
Then we have $b_1>a_0a_7^m>a_7^{17}n^{10}$ by assumption. Let
$$
b_2=a_0 \cdot \prod_{j=m+8}^{d} a_j.
$$
Now $(A,B)$ with $A = \{ a_{5}, a_{6}, a_{7} \}$ and $B = \{ b_{1}, b_{2} \}$ forms a bipartite Diophantine tuple with property $BD_2(n)$ with $a_7\geq 7$ and $b_{1} > a_{7}^{17}n^{10}$. Then Proposition~\ref{prop:antigap} implies that $b_2<b_1^{120}$. In particular, 
$$
a_0 a_{m+7}^{d-m-7}\leq b_2<b_1^{120}<a_0^{120} a_{m+7}^{120m}.
$$
It follows that
$$
d-m-7 \leq 120m+ 120 \frac{\log a_0}{\log a_{m+7}}\leq 120m+120 \frac{\log a_0}{\log a_7}.
$$
We conclude that \[d\ll 1+m+\frac{\log a_0}{\log a_7}\ll 1+\frac{\log a_0}{\log a_7}+\frac{\log \max \{n^{10}/a_0,1\}}{\log a_7}.\qedhere \]
\end{proof}

\section{Bounds on bipartite Diophantine tuples}\label{sec:main1}

In this section, we prove Theorem~\ref{thm1} and Theorem~\ref{thm:ABC}.

\subsection{Proof of Theorem~\ref{thm1}}
Let $\delta=\ep/2$. Let $B=\{b_1, b_2,\ldots, b_m\}$. We may assume that $m\geq 7$, for otherwise we are done. 

First, we consider the case that $k\geq 4$. Partition $B$ into the following 3 subsets:
$$
B_1=B \cap [0, 2|n|^{\frac{k}{k-3}+\delta}], B_2=B \cap (2|n|^{\frac{k}{k-3}+\delta}, 2|n|^{17}], B_3=B \cap (2|n|^{17},\infty).
$$
Since $a_2<b_1$, we have $|B_3|\leq 6$ by Proposition~\ref{prop:a1a2}. Using the gap principle in Corollary~\ref{cor:gap}, it is easy to verify that $|B_2|\ll_{\delta,k} 1$. By Proposition~\ref{prop:sieve2} applied to $M=N=2|n|^{\frac{k}{k-3}+\delta}$, we have $$|B_1|\ll_{\delta,k} |n|^{\frac{\phi(k)(\frac{k}{k-3}+\delta)}{k^2}+\delta}\ll_{\delta,k} |n|^{\frac{\phi(k)}{(k-3)k}+\ep}.$$
Thus, $|B|=|B_1|+|B_2|+|B_3| \ll_{k,\ep} |n|^{\frac{\phi(k)}{(k-3)k}+\ep}$, as required.

Next, we use a similar argument to analyze the case that $k=3$. Partition $B$ into the following 2 subsets:
$$
B_1=B \cap [0, 2|n|^{17}], \quad B_2=B \cap (2|n|^{17},\infty).
$$
Since $a_3<b_1$, we have $|B_2|\leq 6$ by Proposition~\ref{prop:a1a2}. By Proposition~\ref{prop:sieve2} applied to $M=N=2|n|^{17}$, we have $$|B|\ll |B_1|\ll_{\ep} |n|^{\frac{\phi(3) \cdot 17}{3^3}+\ep}\ll_{\ep} |n|^{\frac{34}{27}+\ep},$$
as required.

\subsection{Results conditional on the ABC conjecture: proof of Theorem~\ref{thm:ABC}} 
For a nonzero integer $N$, its \emph{radical} is  $\rad(N)=\prod_{p \mid N}p$. The following is the ABC conjecture.

\begin{conj}[ABC conjecture]\label{conj:ABC}
For each $\ep>0$, there is a constant $K_\ep$, such that whenever $a,b,c$ are nonzero integers with $\gcd(a,b,c)=1$ and $a+b=c$, we have 
$$
\max \{|a|,|b|,|c|\} \leq K_{\ep} \rad (abc)^{1+\ep}.
$$    
\end{conj}

For our purpose, the following folklore consequence of the ABC conjecture is more convenient to apply.
\begin{conj}\label{conj:ABC2}
For each $\ep>0$, there exists a constant $K_\ep$, such that whenever $a,b,c$ are nonzero integers with $a+b=c$, we have 
$$
\max \{|a|,|b|\} \leq K_{\ep} \rad (ab)^{1+\ep} \cdot |c|^{1+\ep}.
$$    
\end{conj}

%For completeness, we include a short proof that Conjecture~\ref{conj:ABC} implies Conjecture~\ref{conj:ABC2}.
%\begin{proof}[Proof of Conjecture~\ref{conj:ABC2} assuming Conjecture~\ref{conj:ABC}]Let $d=\gcd(a,b)$ and set $a'=a/d,b'=b/d,c'=c/d$. Then we have $a'+b'=c'$ and $\gcd(a',b',c')=1$, so Conjecture~\ref{conj:ABC} implies that for each $\ep>0$,
%\begin{align*}
%\frac{\max \{|a|,|b|\}}{d}
%&=\max \{|a'|,|b'|\} \leq K_{\ep} \rad (a'b'c')^{1+\ep}\\
%&\leq K_{\ep} \rad(ab)^{1+\ep} \cdot |c'|^{1+\ep}< K_{\ep} \rad(ab)^{1+\ep} \cdot \frac{|c|^{1+\ep}}{d},
%\end{align*}
%as required. 
%\end{proof}

Next, we use Conjecture~\ref{conj:ABC2} to deduce some bounds on bipartite Diophantine tuples.

\begin{lem}\label{lem:b}
Assume Conjecture~\ref{conj:ABC2}. Let $n$ be a nonzero integer and $k\geq 3$. Assume that $\alpha_1,\alpha_2,\beta$ are positive integers such that $\alpha_1<\alpha_2$, and both $\alpha_1\beta+n$ and $\alpha_2\beta+n$ are $k$-th powers. Then for each real numbers $L_1>\frac{k}{k-2}$ and $L_2>\frac{2k+1}{k-2}$, we have $\beta<C_{k,L_1,L_2}|n|^{L_1}\alpha_2^{L_2}$, where $C_{k,L_1, L_2}$ is a constant depending only on $k$, $L_1$ and $L_2$.
\end{lem}
\begin{proof}
Let $\ep \in (0,\frac{1}{2})$. By assumption, we can find positive integers $x, y$ such that
$$
\alpha_1 \beta+n=x^k, \quad \alpha_2 \beta+n=y^k,
$$
It follows that
\[
\alpha_2 x^k-\alpha_1 y^k=n\left(\alpha_2-\alpha_1\right).    
\] 
Thus, Conjecture~\ref{conj:ABC2} implies that
\begin{equation}\label{eq:eq1}
\alpha_2x^k \leq K_{\ep} \rad(\alpha_1\alpha_2x^ky^k)^{1+\ep} (|n|(\alpha_2-\alpha_1))^{1+\ep} \leq K_{\ep} (\alpha_1\alpha_2xy)^{1+\ep} (|n|\alpha_2)^{1+\ep}
\end{equation}
We may assume that $\beta\geq 2|n|$, for otherwise we are done. Then we have $x^k=\alpha_1\beta+n\in [\alpha_1\beta/2, 2\alpha_1\beta]$ and $y^k=\alpha_2\beta+n\leq 2\alpha_2\beta$. Inequality~\eqref{eq:eq1} thus implies that
$$
\frac{\alpha_1 \beta}{2} \leq x^k \leq K_{\ep} (\alpha_1\alpha_2xy)^{1+\ep} |n|^{1+\ep} \alpha_2^{\ep} < K_{\ep} \alpha_2^{2+3\ep}|n|^{1+\ep} (xy)^{1+\ep} < K_{\ep} \alpha_2^{2+3\ep}|n|^{1+\ep} (4\alpha_1\alpha_2 \beta^2)^{\frac{1+\ep}{k}}.
$$
Since $1+\ep<k$, it follows that
$$
 \frac{\beta}{2} < K_{\ep} \alpha_2^{2+3\ep}|n|^{1+\ep} (4\alpha_2 \beta^2)^{\frac{1+\ep}{k}}.
$$
that is, 
$$
\beta^{1-\frac{2(1+\ep)}{k}}< 2\cdot 4^{\frac{1+\ep}{k}}K_{\ep}|n|^{1+\ep} \alpha_2^{2+3\ep+\frac{1+\ep}{k}},
$$
equivalently,
$$
\beta<(2\cdot 4^{\frac{1+\ep}{k}}K_{\ep})^{\frac{k}{k-2(1+\ep)}}|n|^{\frac{k(1+\ep)}{k-2(1+\ep)}}\alpha_2^{\frac{(2+3\ep)k+1+\ep}{k-2(1+\ep)}}.
$$
Note that $\frac{k(1+\ep)}{k-2(1+\ep)} \to \frac{k}{k-2}$ and $\frac{(2+3\ep)k+1+\ep}{k-2(1+\ep)} \to \frac{2k+1}{k-2}$ as $\ep \to 0$. This completes the proof of the lemma.
\end{proof}

\begin{lem}\label{lem:b1}
Assume Conjecture~\ref{conj:ABC2}. Let $n$ be a nonzero integer and $k\geq 3$. Assume that $\alpha_1,\alpha_2,\beta_1, \beta_2$ are positive integers with $\alpha_1<\alpha_2$ and $\beta_1<\beta_2$,
such that $\alpha_i\beta_j+n$ is a $k$-th powers for $1\leq i,j\leq 2$. Then for all real numbers $L_3>\frac{k^2}{(k-2)^2}$ and $L_4>\frac{6k-2}{(k-2)^2}$, we have $\beta_1<D_{k,L_3, L_4}|n|^{L_3}\alpha_2^{L_4}$, where $D_{k,L_3, L_4}$ is a constant depending only on $k, L_3$, and $L_4$.
\end{lem}
\begin{proof}
Let $\ep \in (0,\frac{1}{2})$. By assumption, we can find positive integers $x_1,x_2,x_3,x_4$ such that
$$
\alpha_1 \beta_1=x_1^k-n, \quad \alpha_1 \beta_2=x_2^k-n, \quad \alpha_2\beta_1=x_3^k-n, \quad \alpha_2\beta_2=x_4^k-n.
$$
It follows that
\[
(x_1^k-n)(x_4^k-n)=\alpha_1\alpha_2\beta_1\beta_2=(x_2^k-n)(x_3^k-n),
\] 
equivalently,
$$
x_1^k x_4^k -x_2^k x_3^k=n(x_1^k+x_4^k-x_2^k-x_3^k)=n(\alpha_1\beta_1+\alpha_2\beta_2-\alpha_1\beta_2-\alpha_2\beta_1)=n(\alpha_2-\alpha_1)(\beta_2-\beta_1).
$$
Thus, Conjecture~\ref{conj:ABC2} implies that
\begin{equation}\label{eq:eq2}
x_1^k x_4^k <K_{\ep} \rad(x_1^kx_2^kx_3^kx_4^k)^{1+\ep} (|n|\alpha_2\beta_2)^{1+\ep} \leq K_{\ep} (x_1x_2x_3x_4)^{1+\ep} (|n|\alpha_2\beta_2)^{1+\ep}.
\end{equation}
We may assume that $\beta_1\geq 2|n|$, for otherwise we are done. Then we have $$x_1^k \in [\alpha_1\beta_1/2, 2\alpha_1\beta_1], \quad x_2^k\leq 2\alpha_1\beta_2, \quad x_3^k \leq 2\alpha_2\beta_1, \quad x_4^k \in [\alpha_2\beta_2/2, 2\alpha_2\beta_2].$$ Inequality~\eqref{eq:eq2} then implies that
$$
\frac{\alpha_1\beta_1\alpha_2\beta_2}{4} \leq x_1^k x_4^k \leq K_{\ep} (x_1x_2x_3x_4)^{1+\ep} (|n|\alpha_2\beta_2)^{1+\ep} \leq K_{\ep} (4\alpha_1\alpha_2\beta_1\beta_2)^{\frac{2(1+\ep)}{k}} (|n|\alpha_2\beta_2)^{1+\ep},
$$
which implies that
\begin{equation}\label{eq:beta1}
\beta_1\leq 4K_{\ep} |n|^{1+\ep} (4\beta_1)^{\frac{2(1+\ep)}{k}} (\alpha_2\beta_2)^{\ep+\frac{2(1+\ep)}{k}}.
\end{equation}
Let $L_1$ and $L_2$ be real numbers with $L_1>\frac{k}{k-2}$ and $L_2>\frac{2k+1}{k-2}$. It follows from Lemma~\ref{lem:b} that  $\beta_2\leq C_{k,L_1, L_2} |n|^{L_1}\alpha_2^{L_2}$. Thus, inequality~\eqref{eq:beta1} implies that
$$
\beta_1\leq 4^{1+\frac{2(1+\ep)}{k}}K_{\ep} C_{k,L_1,L_2}^{\ep+\frac{2(1+\ep)}{k}} |n|^{1+\ep+L_1(\ep+\frac{2(1+\ep)}{k})} \beta_1^{\frac{2(1+\ep)}{k}} \alpha_2^{(\ep+\frac{2(1+\ep)}{k})(L_2+1)},
$$
equivalently,
$$
\beta_1\leq (4^{1+\frac{2(1+\ep)}{k}}K_{\ep} C_{k,L_1,L_2}^{\ep+\frac{2(1+\ep)}{k}})^{\frac{k}{k-2(1+\ep)}} |n|^{\frac{k(1+\ep)+L_1(k\ep+2(1+\ep))}{k-2(1+\ep)}} \alpha_2^{\frac{(k\ep+2(1+\ep))(L_2+1)}{k-2(1+\ep)}}.
$$
Note that as $\ep \to 0$, $L_1 \to \frac{k}{k-2}$, and $L_2 \to \frac{2k+1}{k-2}$, we have $$\frac{k(1+\ep)+L_1(k\ep+2(1+\ep))}{k-2(1+\ep)} \to \frac{k^2}{(k-2)^2}, \quad \frac{(k\ep+2(1+\ep))(L_2+1)}{k-2(1+\ep)} \to \frac{6k-2}{(k-2)^2}.$$
This completes the proof of the lemma.
\end{proof}

Now we are ready to present the proof of Theorem~\ref{thm:ABC}.
\begin{proof}[Proof of Theorem~\ref{thm:ABC}]
Partition $B$ into two subsets:
$$
B_1=B \cap (0, 2|n|^{\theta_k}), \quad B_2=B \cap [2|n|^{\theta_k},+\infty],
$$
where
$$
\theta_3=48, \quad \theta_4=12, \quad \theta_5=\frac{16}{5}, \quad \text{and } \quad \theta_k=\frac{k^2-k}{k^2-6k+3} \text{ for } k\geq 6.
$$

Let $\ep \in (0,1)$. Choose $\delta>0$ sufficiently small, and set 
$$
L_1=\frac{k}{k-2}+\delta,\quad L_2=\frac{2k+1}{k-2}+\delta,\quad L_3=\frac{k^2}{(k-2)^2}+\delta,\quad L_4=\frac{6k-2}{(k-2)^2}+\delta,
$$
such that:
\begin{enumerate}
    \item If $k\geq 6$, we have
    \begin{equation}\label{eq:LLL}
k-1-L_2>1, \quad (k-1-L_2)^4\geq 9>7\geq L_2L_4, \quad \frac{k+L_1}{k-2-\ep-L_2}\geq \theta_k.
\end{equation}
This is possible since when $k\geq 6$, we have
$$
k-1-\frac{2(k+1)}{k-2}>1, \quad \bigg(k-1-\frac{2(k+1)}{k-2}\bigg)^4>9, \quad \frac{(2k+1)(6k-2)}{(k-2)^{3}}<7
$$
and 
$$
\frac{k+\frac{k}{k-2}}{k-2-\frac{2k+1}{k-2}}=\frac{k^2-k}{k^2-6k+3}=\theta_k.
$$
   \item If $k\in \{4,5\}$, we have
    \begin{equation}\label{eq:LLL45}
k-1-\frac{L_2}{k-1}>1, \quad \bigg(k-1-\frac{L_2}{k-1}\bigg)^{12}\geq 30>25\geq L_2L_4, \quad \frac{k+L_1}{k-2-\ep-\frac{L_2}{k-1}}\geq \theta_k.
\end{equation}
This is possible since when $k\in \{4,5\}$, we have
$$
k-1-\frac{2(k+1)}{(k-1)(k-2)}>1, \quad \bigg(k-1-\frac{2(k+1)}{(k-1)(k-2)}\bigg)^{12}>30, \quad \frac{(2k+1)(6k-2)}{(k-2)^{3}}<25.
$$
   \item If $k=3$, we have
    \begin{equation}\label{eq:LLL3}
\bigg(k-1-\frac{L_2}{(k-1)^3}\bigg)^{41}\geq 125>113>L_2L_4, \quad \frac{k+L_1}{k-2-\ep-\frac{L_2}{(k-1)^3}}\geq \theta_3.
\end{equation}
\end{enumerate}
Note that the above choices of $L_1,L_2,L_3,L_4$ only depend on $k$ and $\ep$.

We claim that $|B_1|\ll_{k,\ep} |n|^{\frac{\theta_k \phi(k)}{k^{|A|}}+\ep}$ and $|B_2|\ll_{k,\ep} 1$, and thus the desired upper bound on $|B|$ follows. Indeed, by Lemma~\ref{lem:b}, we can apply Proposition~\ref{prop:sieve2} with $$N=2|n|^{\theta_k}, \quad \text{and} \quad M=C_{k,L_1,L_2}|n|^{L_1}(2|n|^{\theta_k})^{L_2}$$ to deduce the desired bound on $|B_1|$.

It remains to bound $|B_2|$. Let $A=\{a_1,a_2,\ldots, a_t\}$ with $a_1<a_2<\cdots<a_t$, and $B_2=\{b_1,b_2,\ldots, b_m\}$ with $b_1<b_2<\ldots<b_m$. We may assume that $m\geq 100$, for otherwise we are done. By Lemma~\ref{lem:b} applied to $\alpha_1=b_1$, $\alpha_2=b_2$, and $\beta=a_t$, we deduce that
\begin{equation}\label{eq:a2}
a_t\ll_{k,\ep} |n|^{L_1} b_2^{L_2}.    
\end{equation}
By Lemma~\ref{lem:b1} applied to $\alpha_1=a_1$, $\alpha_2=a_2$, $\beta_1=b_{m-1}$, and $\beta_2=b_m$, we deduce that $b_{m-1}\ll_k|n|^{L_3} a_2^{L_4}$. It then follows from inequality~\eqref{eq:a2} that
\begin{equation}\label{eq:bm}
b_{m-1}\ll_{k,\ep} |n|^{L_3} a_2^{L_4} \ll_{k,\ep}|n|^{L_3+L_1L_4} b_2^{L_2L_4}.
\end{equation}
On the other hand, since $a_2b_2\geq 2|n|$, for each $2\leq i \leq m-2$, by Lemma~\ref{gap_principle} applied to $x=a_1, y=a_2,z=b_i, w=b_{i+1}$, we deduce from inequality~\eqref{eq:a2} that
\begin{equation}\label{eq:gap}
b_{i+1}\geq \frac{k^k}{4^{k-1}|n|^k} \frac{(a_1b_i)^{k-1}}{a_2} \gg_{k} \frac{b_i^{k-1}}{a_2 |n|^k}.
\end{equation}

Next, we consider three cases based on the size of $k$.

(1) $k\geq 6$. Combining inequalities~\eqref{eq:a2} and~\eqref{eq:gap}, we have
\begin{equation}\label{eq:gap6}
b_{i+1}\gg_{k,\ep} \frac{b_i^{k-1-L_2}}{|n|^{k+L_1}}.
\end{equation}
Applying inequality~\eqref{eq:gap6} four times and combining inequality~\eqref{eq:LLL}, we have
$$
b_{m-1}\gg_{k,\ep} \frac{b_{m-5}^{9}}{|n|^L}
$$
for some constant $L$ depending only on $k$ and $\ep$. It then follows from inequalities~\eqref{eq:LLL} and~\eqref{eq:bm} that
$$
b_{m-5}^{9}\ll_{k,\ep} b_2^7 |n|^{L'}.
$$
for some constant $L'$ depending only on $k$ and $\ep$. It follows that $b_{m-5}\ll_{k,\ep} |n|^{L'/2}$. On the other hand, for each $2\leq i \leq m-2$, since $b_i\geq 2|n|^{\frac{k^2-k}{k^2-6k+3}}$, inequalities~\eqref{eq:LLL} and~\eqref{eq:gap6} imply that
\begin{equation}\label{eq:ggg}
b_{i+1}\gg_{k,\ep} \frac{b_i^{k-1-L_2}}{|n|^{k+L_1}} \gg_{k,\ep} b_i^{k-1-L_2-\frac{(k+L_1)(k^2-6k+3)}{k^2-k}}\geq b_i^{k-1-L_2-(k-2-\ep-L_2)}=b_i^{1+\ep}. 
\end{equation}
However, we have $b_1\geq 2|n|$ and $b_{m-5}\ll_{k,\ep} |n|^{L'/2}$; inequality~\eqref{eq:ggg} thus forces $m\ll_{k,\ep} 1$, as required.

(2) $k\in \{4,5\}$. By Lemma~\ref{gap_principle} with $x=a_2$, $y=a_3$, $z=b_2$, and $w=b_3$, we have 
\begin{equation}\label{eq:a_3b_3}
a_3 \geq \frac{a_2^{k-1}}{|n|^k}\frac{b_2^{k-1}}{b_3}.
\end{equation}
We consider the following two cases.

Case 1: $b_2^{k-1}\geq b_3$. In this case inequality~\eqref{eq:a_3b_3} implies that $a_3\geq a_2^{k-1}/|n|^k$. It follows from inequality~\eqref{eq:a2} that $a_2^{k-1}\leq a_3|n|^k \ll_{k,\ep} |n|^{L_1+k} b_2^{L_2}$ and thus $a_2\ll_{k,\ep} |n|^{\frac{L_1+k}{k-1}} b_3^{\frac{L_2}{k-1}}.$

Case 2: $b_3 \geq b_2^{k-1}$. In this case, inequality~\eqref{eq:a2} implies that 
$a_2\ll_{k,\ep} |n|^{L_1}b_2^{L_2} \leq |n|^{L_1}b_3^{\frac{L_2}{k-1}}$. 

Since $\frac{L_1+k}{k-1}<L_1$, in both cases we have 
\begin{equation}\label{eq:a2k45}
a_2\ll_{k,\ep} |n|^{L_1}b_3^{\frac{L_2}{k-1}}. 
\end{equation}
Combining inequalities~\eqref{eq:a2} and~\eqref{eq:gap}, for each $2\leq i \leq m-1$, we have
\begin{equation}\label{eq:gap45}
b_{i+1}\gg_{k,\ep} \frac{b_i^{k-1-\frac{L_2}{k-1}}}{|n|^{k+L_1}}.
\end{equation}
Now we can proceed similarly as in the case of $k\geq 6$, to conclude that $|B_2|\ll_{k,\ep} 1$. More precisely, we apply inequality~\eqref{eq:gap45} for $m-13\leq i \leq m-2$ and combine that with inequalities~\eqref{eq:LLL45} and~\eqref{eq:bm} to show $b_{m-13}\ll_{k,\ep} |n|^L$ for some constant $L$ depending only on $k$ and $\ep$, and show that $b_{i+1}\gg_{k,\ep} b_i^{1+\ep}$ for each $2\leq i \leq m-13$; since $b_1\geq |n|^{\theta_k}$, by comparing these two estimates, we immediately have $m\ll_{k,\ep} 1$.

(3) $k=3$. In this case, we can apply the same argument as in the case $k\in \{4,5\}$ three times to obtain that 
$$
a_2\ll_{\ep} |n|^{L_1} b_5^{\frac{L_2}{(k-1)^3}}.
$$
Now we can use inequality~\eqref{eq:LLL3} and proceed similarly as in the previous cases, to conclude that $|B_2|\ll_{\ep} 1$. 
\end{proof}

\section{Bounds on variants of Diophantine tuples}\label{sec:app}
In this section, we combine all the ingredients we have so far to prove Theorems~\ref{thm:prodl},~\ref{thm:a0=1}, and~\ref{thm:a0small}. 

\begin{proof}[Proof of Theorem~\ref{thm:prodl}]
If $|A|< 4\ell$, then we are done. Next, assume that $|A|\geq 4\ell$. Partition $A$ into two subsets $A_1$ and $A_2$ such that $|A_1|=|A_2|$ or $|A_1|=|A_2|+1$. Then we have $|A_1|,|A_2|\geq 2\ell$. Let $h=\lfloor \frac{\ell}{2}\rfloor$. Let $B_1=A_1^{(h)}$ and $B_2=A_2^{(\ell-h)}$. Observe now that $(a_0B_1,B_2)$ is a bipartite Diophantine tuple with property $BD_k(n)$. By Lemma~\ref{lem:hfold}, we have
$$
|B_1|\geq h(|A_1|-h)+1 \gg h|A_1|\gg \ell|A|.
$$
Similarly, $|B_2|\gg \ell |A|$. On the other hand, by Theorem~\ref{thm:logn}, we have $\min \{|B_1|, |B_2|\}\ll_{k} \log (|n|+1)$. It follows that $|A|\ll_{k} \frac{\log (|n|+1)}{\ell}$, as required.
\end{proof}

\begin{proof}[Proof of Theorem~\ref{thm:a0=1}]
Let $A=\{a_1,a_2,\ldots, a_m\}$ with $a_1<a_2<\cdots<a_m$. Since $1+n$ and $a_2+n$ are both $k$-th powers and $a_2\geq 2$, it follows that
$$
a_2+n\geq \left((1+n)^{1/k}+1\right)^k\geq 1+n+k(1+n)^{(k-1)/k}.
$$
It follows that $a_2\geq kn^{(k-1)/k}$. Next, we consider two cases.

First, consider the case $k=2$. Assume that $|A|\geq 133$. For each $1\leq i \leq 22$, let $b_i=\prod_{j=6i-4}^{6i+1}a_i$. Then we have $n^3<b_1<b_2<\cdots<b_{22}$ and $\{b_1, b_2,\ldots, b_{22}\}$ forms a Diophantine tuple with property $D_2(n)$. However, Dujella \cite{D02} showed that such a Diophantine tuple with property $D_2(n)$ (with all elements at least $n^3$) has size at most $21$, a contradiction. Thus, $|A|\leq 132$.

Finally, consider the case $k\geq 3$. Assume that $|A|\geq 10$. Then by Corollary~\ref{cor:a0X}, we have $\prod_{j=6}^{10} a_j\leq n^{k/(k-2)}\leq n^3$. On the other hand, we have $\prod_{j=6}^{10} a_j\geq (kn^{(k-1)/k})^5>n^3$, a contradiction. Thus, $|A|\leq 9$.
\end{proof}

\begin{proof}[Proof of Theorem~\ref{thm:a0small}]
We may assume that $a_1<a_2<\ldots<a_d$. Since $n+a_0$ and $n+a_2a_0$ are both $k$-th powers and $a_2\geq 2$, we have 
$$
n+a_2a_0\geq ((n+a_0)^{1/k}+1)^k\geq n+a_0+k(n+a_0)^{(k-1)/k}.
$$
It follows that $a_2a_0\geq k(n+a_0)^{(k-1)/k}$ and thus $a_2a_0\geq kn^{(k-1)/k}$. Since $a_0\leq n^{\frac{k-1}{k}-\ep}$, it follows that $a_2\geq n^{\ep}$. Next, we consider two cases.

When $k=2$, we may assume that $d\geq 7$. It follows from Corollary~\ref{cor:loga0} that $$d\ll 1+\frac{\log n}{\log a_7}\leq 1+\frac{\log n}{\log a_2}\ll \frac{1}{\ep},$$ where the implied constant is absolute.

Finally, consider the case $k\geq 3$. We may assume that $d$ is even and $d\geq 10$. Let $X=\prod_{j=d/2+1}^{d} a_j$; then $X>n^{\ep d/2}$. By Corollary~\ref{cor:a0X}, we have $a_0X<n^{k/(k-2)}\leq n^3$. It follows that $d\ll \frac{1}{\ep}$, where the implied constant is absolute.
\end{proof}

\section{Proof of Theorem~\ref{thm:a1}: simultaneous Pell equations}\label{sec:Pell}
In this section, we prove Theorem~\ref{thm:a1} via simultaneous Pell equations. 

Let $a$ be a positive integer and $u$ be a nonzero integer. Recall that every integral solution to the (generalized) Pell equation $x^{2} - az^{2} = u$ takes the form $\alpha^{k}\mu$, where $\alpha$ is the fundamental solution to
\begin{equation*}
    x^{2} - az^{2} = 1
\end{equation*}
and $\mu$ is one of the base solutions to 
\begin{equation*}
    x^{2} - az^{2} = u .
\end{equation*}
Here, a base solution means a minimal positive solution corresponding to an integer $l$ with $l^{2} \equiv a \pmod{u}$, where $0 \leq l < \abs{u}$. It is well-known that there are at most $2^{\omega(u)}$ different base solutions \cite[Section 6]{B98}, where $\omega(u)$ is the number of distinct prime factors of $u$.

We have the following gap principle for solutions to simultaneous Pell equations by Bennett \cite{B98}.
\begin{lem}[{\cite[Lemma 6.1]{B98}}]\label{lem:Bennett_thm_in_gap_principle}
Let $a,b$ be positive integers and $u,v$ be nonzero integers such that $av\neq bu$. Let $(x_{i}, y_{i}, z_{i})$ be positive integral solutions to
\begin{align*}
    x^{2} - az^{2} = u, \qquad y^{2}-bz^{2} = v
\end{align*}
for $1\leq i\leq 3$, belonging to a fixed pair of base solutions. If $z_{3} >z_{2} >z_{1} >\max \{ \abs{u}, \abs{v} \}^{3}$, then $z_{3} >z_{1}^{3}$. 
\end{lem}

Lemma~\ref{lem:Bennett_thm_in_gap_principle} leads to the following gap principle for bipartite Diophantine tuples. Although this is not needed to prove Theorem~\ref{thm:a1}, we include it here since it might be of independent interest.

\begin{prop}
Suppose that $(A,B)$ with $A = \{ a_{1}, a_{2}, a_{3} \}$ and $B = \{ b_{1}, b_{2}, \dots, b_{m} \}$ forms a bipartite Diophantine tuple with property $BD_2(n)$, where $a_1<a_2<a_3$, $b_1<b_2<\cdots<b_m$, and $n$ is a non-zero integer. 
Let
$$h = 2^{\min \{ \omega(na_{1}(a_{1}-a_{2})), \  \omega(na_{1}(a_{1}-a_{3})) \}}.$$
If $b_{1} > n^{6}a_{1}^{5}a_{3}^{6}$ and $m \geq 2h+1$, then $b_{m} > ( 1 - \frac{1}{3^{5}}) a_{1}^{2}b_{1}^{3}$.
\end{prop}
\begin{proof}
For each $1 \leq i \leq m$, by definition, there are positive integers $x_{i},y_{i}, z_{i}$ such that
\begin{align}\label{eqn:BDn_defn}
    a_{1}b_{i}+n &= z_{i}^{2}, & a_{2}b_{i}+n &= y_{i}^{2}, & a_{3}b_{i}+n &= x_{i}^{2}. 
\end{align} 
Note that each $(x_i,y_i,z_i)$ is a solution of the following system of Pell equations
\begin{align*}
    a_{1}x^{2} - a_{3}z^{2} &= n(a_{1}-a_{3}), \\
    a_{1}y^{2} - a_{2}z^{2} &= n(a_{1}-a_{2}) .
\end{align*}
Observe that if $(x,y,z)$ is a solution to the above system of Pell equations, then $(a_{1}x,y,z)$ will be the solution of the following system of Pell equations
\begin{align}\label{eqn:simultaneous_generalized_Pell_1}
\begin{aligned}
    x^{2} - a_{1}a_{3}z^{2} &= na_{1}(a_{1}-a_{3}), \\
    y^{2} - a_{1}a_{2}z^{2} &= na_{1}(a_{1}-a_{2}) .
\end{aligned}
\end{align}
Thus, each $(a_{1}x_{i},y_{i},z_{i})$ is a solution to the system of Pell equations in~\eqref{eqn:simultaneous_generalized_Pell_1}.

Since $m\geq 2h+1$ and there are at most $h$ different base solutions to the system of simultaneous Pell equations~\eqref{eqn:simultaneous_generalized_Pell_1}, the pigeonhole principle guarantees that there are at least three solutions to \eqref{eqn:simultaneous_generalized_Pell_1} that arise from the same base solution.
Denote the three solutions by $(x_{i_{j}}, y_{i_{j}}, z_{i_{j}})$ for $1 \leq j \leq 3$, with $i_1<i_2<i_3$. Since
$b_{1} > n^{6}a_{1}^{5}a_{3}^{6}$, we have
$$z_{i_1}^2\geq z_{1}^{2} = a_{1}b_{1}+n > n^{6}a_{1}^{6}a_{3}^{6} + n \geq n^{6}a_{1}^{6}a_{3}^{6} - \abs{n} \geq n^{6}a_{1}^{6}a_{3}^{6} - n^{6}a_{1}^{6} > n^{6}a_{1}^{6}(a_{1}-a_{3})^{6}.$$

Applying Lemma~\ref{lem:Bennett_thm_in_gap_principle} gives us $z_{i_{3}} > z_{i_{1}}^{3}$. In particular, $z_{m} \geq z_{i_{3}} > z_{i_{1}}^{3} \geq z_{1}^{3}$, that is,
\begin{align*}
    a_{1}b_{m} +n > (a_{1}b_{1}+n)^{3} .
\end{align*}
We then deduce that
$$
a_{1}b_{m} > a_{1}^{3}b_{1}^{3} -3\abs{n}a_{1}^{2}b_{1}^{2} = a_{1}^{2}b_{1}^{2}(a_{1}b_{1} - 3\abs{n}) > a_{1}^{2}b_{1}^{2} \left( a_{1}b_{1} - \frac{3a_{1}b_{1}}{\abs{n}^{5}a_{1}^{6}a_{3}^{6}} \right) > a_{1}^{3}b_{1}^{3} \left( 1- \frac{3}{a_{3}^{6}} \right) .
$$
The desired conclusion $b_{m} > ( 1 - \frac{1}{3^{5}}) a_{1}^{2}b_{1}^{3}$ now follows from $a_{3} \geq 3$.
\end{proof}

We proceed to deduce a gap principle for multiplicative Hilbert cubes contained in shifted squares. 

\begin{thm}\label{thm:explicit_bound_d}
Let $n$ be a nonzero integer. Let $a_0, a_1,\ldots, a_d$ be positive integers such that $2\leq a_{1} < a_{2} < \cdots < a_{d}$. Set $$m = 2^{\min \{ \omega(n(1-a_{1})),\ \omega(n(1-a_{2})) \}}.$$ Assume that $H^\times (a_0;a_1,a_2,\ldots, a_d) \subseteq \{x^2-n: x \in \N\}$. 
If $a_{0} > n^{6} a_{2}^{4}$, then $d < \frac{3+ \sqrt{32m+17}}{2}$. 
\end{thm}
\begin{proof}
Suppose otherwise that $d \geq \frac{3+ \sqrt{32m+17}}{2}$. Set 
\begin{align}\label{eqn:subset_of_solutions_b}
\begin{aligned}
    B=&\left\{ a_{i} \prod_{j = d-\ell+3}^{d} a_{j} \mid 3 \leq i \leq d-\ell+2 \right\} \bigcup \left\{ a_{i} \prod_{j = d-\ell+2}^{d} a_{j} \mid 3 \leq i \leq d-\ell+1 \right\} \\
    &\bigcup \cdots \bigcup \left\{ a_{i} \prod_{j = d-\ell-k +4}^{d} a_{j} \mid 3 \leq i \leq d-\ell - k +3 \right\},
\end{aligned}
\end{align}
where $\ell$ is a positive integer, to be specified later, satisfying
$$
\frac{2d+13 - \sqrt{(2d+13)^{2}-16(4d+2m+10)}}{8} < \ell < \frac{2d+13 + \sqrt{(2d+13)^{2}-16(4d+2m+10)}}{8}
$$
and $k = 2\ell -4$.
Note that our assumption $d \geq \frac{3+ \sqrt{32m+17}}{2}$ guarantees that $(2d+13)^{2}-16(4d+2m+10) \geq 17$. There must exist a positive integer $\ell$ lying in the desired range. We then pick $\ell$ to be the smallest positive integer in that range.

By definition, for each $b\in B$, there are positive integers $x,y$ and $z$ such that
\begin{align*}
    a_{0}b + n &= z^{2} & a_{0}a_{1}b + n &= x^{2} & a_{0}a_{2}b + n &= y^{2},
\end{align*}
which corresponds to a solution to the following system of simultaneous Pell equations
\begin{align}\label{eqn:simultaneous_Pell_Hilbert_cubes_1}
    x^{2} - a_{1}z^{2} &= n(1-a_{1}), & y^{2} - a_{2}z^{2} &= n(1-a_{2}) .
\end{align}

Note that $|B|=\frac{(2d-2\ell-k+1) k}{2}$, where  each element corresponds to a solution to the system of simultaneous Pell equations~\eqref{eqn:simultaneous_Pell_Hilbert_cubes_1}. Moreover, the solutions are arranged in increasing order. From our choice of $k$ and $\ell$, we see that $\frac{(2d-2\ell-k+1) k}{2} > 2m$. Thus, by the pigeonhole principle, there are at least three solutions coming from the same base solution. We check that
\begin{align*}
    a_{0}a_{3} \prod_{j=d-\ell +3}^{d} a_{j} \ + n > n^{6}a_{2}^{4}a_{3} \prod_{j=d-\ell +3}^{d} a_{j} \ - \abs{n} > n^{6}a_{2}^{4}a_{3}a_{d} - \abs{n} > n^{6}(1-a_{2})^{6} .
\end{align*}
Hence, we can apply Lemma~\ref{lem:Bennett_thm_in_gap_principle} to deduce that 
\begin{align*}
    a_{0} \prod_{j=d-\ell-k+3}^{d} a_{j} \ + n > \bigg(a_{0}a_{3} \prod_{j=d-\ell +3}^{d} a_{j} \ + n\bigg)^{3} .
\end{align*}
We then deduce that
\begin{align*}
    a_{0}\prod_{j=d-3\ell+7}^{d} a_{j} &> a_{0}^{3}a_{3}^{3} \left(\prod_{j=d-\ell+3}^{d} a_{j}^{3} \right) -3\abs{n}a_{0}^{2}a_{3}^{2} \left(\prod_{j=d-\ell+3}^{d} a_{j}^{2} \right) \\
    &= a_{0}^{2}a_{3}^{2} \left(\prod_{j=d-\ell+3}^{d} a_{j}^{2} \right) \left( a_{0}a_{3} \left(\prod_{j=d-\ell+3}^{d} a_{j} \right) -3\abs{n} \right) \\
    &> a_{0}^{3}a_{3}^{3} \prod_{j=d-\ell+3}^{d} a_{j}^{3} \left( 1 - \frac{3}{\abs{n}^{5}a_{2}^{4}a_{3}} \right) > \frac{15}{16} a_{0}^{3}a_{3}^{3} \prod_{j=d-\ell+3}^{d} a_{j}^{3}.
\end{align*}
Rearrange terms, we deduce that $$\prod_{j=d-3\ell+7}^{d-\ell+2} a_{j} > \frac{15}{16} a_{0}^{2}a_{3}^{3} \prod_{j=d-\ell+3}^{d} a_{j}^{2}>\prod_{j=d-\ell+3}^{d} a_{j}^{2},$$
violating the assumption that $a_1<a_2<\cdots<a_d$.
\end{proof}

Finally, we combine Corollary~\ref{cor:loga0} and Theorem~\ref{thm:explicit_bound_d} to conclude Theorem~\ref{thm:a1}.

\begin{proof}[Proof of Theorem~\ref{thm:a1}]
If $a_0\leq n^{6} a_{2}^{4}$, Corollary~\ref{cor:loga0} implies that $d \ll \log (|n|+1)$ and we are done. Next assume that $a_{0} > n^{6} a_{2}^{4}$. The theorem then follows from Theorem~\ref{thm:explicit_bound_d} and the estimate
\[
\omega(n(1-a_1))\leq \frac{(1+o(1))\log (|n|a_1)}{\log \log (|n|a_1)}.\qedhere
\]
\end{proof}

\section*{Acknowledgments}
The authors thank Michael Bennett, Ernie Croot, Andrej Dujella, Greg Knapp, and Attila Pethő for helpful discussions. The authors also thank the anonymous referees for their valuable comments and suggestions. The first author was supported in part by an NSERC Discovery Grant. The second author was supported in part by an NSERC fellowship.

\bibliographystyle{abbrv}
\bibliography{main}

\begin{thebibliography}{10}

\bibitem{BLS09}
W.~D. Banks, F.~Luca, and L.~Szalay.
\newblock A variant on the notion of a {D}iophantine {$s$}-tuple.
\newblock {\em Glasg. Math. J.}, 51(1):83--89, 2009.

\bibitem{BHP25}
G.~Batta, L.~Hajdu, and A.~Pongr\'acz.
\newblock On {D}iophantine graphs.
\newblock {\em J. Lond. Math. Soc. (2)}, 111(5):Paper No. e70163, 2025.

\bibitem{B98}
M.~A. Bennett.
\newblock On the number of solutions of simultaneous {P}ell equations.
\newblock {\em J. Reine Angew. Math.}, 498:173--199, 1998.

\bibitem{B07}
M.~A. Bennett.
\newblock The {D}iophantine equation {$(x^k-1)(y^k-1)=(z^k-1)^t$}.
\newblock {\em Indag. Math. (N.S.)}, 18(4):507--525, 2007.

\bibitem{BDHL11}
A.~B\'{e}rczes, A.~Dujella, L.~Hajdu, and F.~Luca.
\newblock On the size of sets whose elements have perfect power {$n$}-shifted products.
\newblock {\em Publ. Math. Debrecen}, 79(3-4):325--339, 2011.

\bibitem{BCM22}
N.~C. Bonciocat, M.~Cipu, and M.~Mignotte.
\newblock There is no {D}iophantine {$D(-1)$}-quadruple.
\newblock {\em J. Lond. Math. Soc. (2)}, 105(1):63--99, 2022.

\bibitem{B04}
Y.~Bugeaud.
\newblock On the {D}iophantine equation {$(x^k-1)(y^k-1)=(z^k-1)$}.
\newblock {\em Indag. Math. (N.S.)}, 15(1):21--28, 2004.

\bibitem{BD03}
Y.~Bugeaud and A.~Dujella.
\newblock On a problem of {D}iophantus for higher powers.
\newblock {\em Math. Proc. Cambridge Philos. Soc.}, 135(1):1--10, 2003.

\bibitem{BG04}
Y.~Bugeaud and K.~Gyarmati.
\newblock On generalizations of a problem of {D}iophantus.
\newblock {\em Illinois J. Math.}, 48(4):1105--1115, 2004.

\bibitem{CHM}
L.~Caporaso, J.~Harris, and B.~Mazur.
\newblock Uniformity of rational points.
\newblock {\em J. Amer. Math. Soc.}, 10(1):1--35, 1997.

\bibitem{CY25}
E.~{Croot} and C.~H. {Yip}.
\newblock {Diophantine tuples and product sets in shifted powers}.
\newblock 2026.
\newblock J. Lond. Math. Soc., to appear. arXiv:2504.04354.

\bibitem{DH94}
J.~A. Dias~da Silva and Y.~O. Hamidoune.
\newblock Cyclic spaces for {G}rassmann derivatives and additive theory.
\newblock {\em Bull. London Math. Soc.}, 26(2):140--146, 1994.

\bibitem{DE12}
R.~Dietmann and C.~Elsholtz.
\newblock Hilbert cubes in progression-free sets and in the set of squares.
\newblock {\em Israel J. Math.}, 192(1):59--66, 2012.

\bibitem{DE15}
R.~Dietmann and C.~Elsholtz.
\newblock Hilbert cubes in arithmetic sets.
\newblock {\em Rev. Mat. Iberoam.}, 31(4):1477--1498, 2015.

\bibitem{DKM22}
A.~B. Dixit, S.~Kim, and M.~R. Murty.
\newblock Generalized {D}iophantine {$m$}-tuples.
\newblock {\em Proc. Amer. Math. Soc.}, 150(4):1455--1465, 2022.

\bibitem{D02}
A.~Dujella.
\newblock On the size of {D}iophantine {$m$}-tuples.
\newblock {\em Math. Proc. Cambridge Philos. Soc.}, 132(1):23--33, 2002.

\bibitem{D24}
A.~Dujella.
\newblock {\em {D}iophantine $m$-tuples and {E}lliptic Curves}, volume~79 of {\em Developments in Mathematics}.
\newblock Springer, Cham, 2024.

\bibitem{DS25}
A.~{Dujella} and L.~{Szalay}.
\newblock {Four squares from three numbers}, 2025.
\newblock arXiv:2506.14013.

\bibitem{E83}
J.-H. Evertse.
\newblock {\em Upper bounds for the numbers of solutions of {D}iophantine equations}, volume 168 of {\em Mathematical Centre Tracts}.
\newblock Mathematisch Centrum, Amsterdam, 1983.

\bibitem{G71}
P.~X. Gallagher.
\newblock A larger sieve.
\newblock {\em Acta Arith.}, 18:77--81, 1971.

\bibitem{HS20}
L.~Hajdu and A.~S\'{a}rk\"{o}zy.
\newblock On multiplicative decompositions of polynomial sequences, {III}.
\newblock {\em Acta Arith.}, 193(2):193--216, 2020.

\bibitem{HTZ19}
B.~He, A.~Togb\'{e}, and V.~Ziegler.
\newblock There is no {D}iophantine quintuple.
\newblock {\em Trans. Amer. Math. Soc.}, 371(9):6665--6709, 2019.

\bibitem{H08}
N.~Hegyv\'ari.
\newblock On additive and multiplicative {H}ilbert cubes.
\newblock {\em J. Combin. Theory Ser. A}, 115(2):354--360, 2008.

\bibitem{KK01}
A.~Kihel and O.~Kihel.
\newblock Sets in which the product of any {$K$} elements increased by {$t$} is a {$k$}th-power.
\newblock {\em Fibonacci Quart.}, 39(2):98--100, 2001.

\bibitem{KYY25}
S.~{Kim}, C.~H. {Yip}, and S.~{Yoo}.
\newblock {$f$-Diophantine sets over finite fields via quasi-random hypergraphs from multivariate polynomials}, 2025.
\newblock arXiv:2503.19603.

\bibitem{KYY}
S.~{Kim}, C.~H. {Yip}, and S.~{Yoo}.
\newblock {Multiplicative structure of shifted multiplicative subgroups and its applications to Diophantine tuples}, 2026.
\newblock Canad. J. Math., to appear. arXiv:2309.09124.

\bibitem{LPS67}
L.~J. Lander, T.~R. Parkin, and J.~L. Selfridge.
\newblock A survey of equal sums of like powers.
\newblock {\em Math. Comp.}, 21:446--459, 1967.

\bibitem{LN97}
R.~Lidl and H.~Niederreiter.
\newblock {\em Finite fields}, volume~20 of {\em Encyclopedia of Mathematics and its Applications}.
\newblock Cambridge University Press, Cambridge, second edition, 1997.

\bibitem{S22}
I.~D. Shkredov.
\newblock On sums and products of combinatorial cubes.
\newblock {\em Finite Fields Appl.}, 77:Paper No. 101948, 17, 2022.

\bibitem{Y24+}
C.~H. {Yip}.
\newblock {Improved upper bounds on Diophantine tuples with the property $D(n)$}.
\newblock {\em Bull. Aust. Math. Soc.}, 111(3):428--432, 2025.

\bibitem{Y24}
C.~H. Yip.
\newblock Multiplicatively reducible subsets of shifted perfect {$k$}th powers and bipartite {D}iophantine tuples.
\newblock {\em Acta Arith.}, 218(3):251--271, 2025.

\bibitem{Y26}
C.~H. {Yip}.
\newblock {Multiplicative irreducibility of small perturbations of the set of shifted $k$-th powers}.
\newblock {\em Combinatorica}, 46(1):Paper No. 1, 2026.

\bibitem{YY25}
C.~H. {Yip} and S.~{Yoo}.
\newblock {$F$-{D}iophantine sets over finite fields}.
\newblock {\em Int. J. Number Theory}, 21(5):1043--1050, 2025.

\end{thebibliography}

\end{document}